\documentclass[11pt,twoside,reqno]{amsart}
\usepackage{relsize}
\usepackage{scalerel}
\usepackage[margin=1in]{geometry}
\usepackage{stackengine,wasysym}
\usepackage{todonotes}
\usepackage{xcolor}
\usepackage{mathrsfs}
\usepackage{dsfont}
\usepackage{mathtools}
\mathtoolsset{showonlyrefs}
\usepackage[hyperfootnotes=false,colorlinks=true,linkcolor = blue, urlcolor  = blue, citecolor = blue]{hyperref}
\usepackage[sort]{cite}
\usepackage{float}
\usepackage{amsmath, amsthm, amssymb}
\usepackage{times}
\usepackage{color}
\usepackage{comment}
\usepackage[toc,page]{appendix}
\usepackage{verbatim}
\usepackage{enumitem}

\newcommand{\pa}{\partial}
\newcommand{\la}{\label}
\newcommand{\fr}{\frac}
\newcommand{\na}{\nabla}
\newcommand{\be}{\begin{equation}}
\newcommand{\ee}{\end{equation}}
\newcommand{\ba}{\begin{array}{l}}
\newcommand{\ea}{\end{array}}

\newcommand{\beg}{\begin}

\renewcommand{\l}{\Lambda}

\usepackage{MnSymbol,wasysym}
\newcommand{\N}{\mathbb N}

\newcommand{\R}{\mathbb R}

\def\ZZ{{\mathbb Z}}
\def\RR{{\mathbb R}}
\def\TT{{\mathbb T}}
\def\NN{\mathbb N}

\theoremstyle{plain}

\newtheorem{Thm}{Theorem}[section]

\theoremstyle{definition}

\theoremstyle{remark}

\allowdisplaybreaks

\numberwithin{equation}{section}

\title[Approximations of Nonlocal Periodic Operators]{A PDE Perspective on Approximating Nonlocal Periodic Operators with Applications on Neural Networks  for Critical SQG Equations}

\author[E. Abdo]{Elie Abdo}
\address[E. Abdo]
{	Department of Mathematics \\
     University of California  \\
	Santa Barbara, CA 93106-3080, USA.} \email{elieabdo@ucsb.edu}

\author[R. Hu]{Ruimeng Hu}
\address[R. Hu]
{Department of Mathematics \\
Department of Statistics and Applied Probability \\
     University of California  \\
	Santa Barbara, CA 93106-3080, USA.} \email{rhu@ucsb.edu}

\author[Q. Lin]{Quyuan Lin}
\address[Q. Lin]
{	School of Mathematical and Statistical Sciences \\
Clemson University\\
Clemson, SC 29634, USA.} \email{quyuanl@clemson.edu}
\date{November 18, 2024}

\begin{document}
\maketitle

\begin{abstract} 
Nonlocal periodic operators in partial differential equations (PDEs) pose challenges in constructing neural network solutions, which typically lack periodic boundary conditions. In this paper, we introduce a novel PDE perspective on approximating these nonlocal periodic operators. Specifically, we investigate the behavior of the periodic first-order fractional Laplacian and Riesz transform when acting on nonperiodic functions, thereby initiating a new PDE theory for approximating solutions to equations with nonlocalities using neural networks. Moreover, we derive quantitative Sobolev estimates and utilize them to rigorously construct neural networks that approximate solutions to the two-dimensional periodic critically dissipative Surface Quasi-Geostrophic (SQG) equation.
\end{abstract}

\vspace{0.5cm}


\smallskip


\section{Introduction}
Nonlocal operators appear in many partial differential equations (PDEs) extensively studied in the literature, such as the surface quasi-geostrophic (SQG) equations \cite{caffarelli2010drift, constantin2016critical, constantin2020estimates, constantin2018local, constantin2015long, constantin2012nonlinear, constantin1999behavior, kiselev2007global, ignatova2019construction, stokols2020holder}, the incompressible porous media model \cite{castro2009incompressible}, and fractional Boussinesq systems \cite{wu2014well, wu20182d, yang2014global, yang20183d}. Approximations of solutions to these PDEs do not always belong to the class of functions on which these nonlocal operators are defined and studied. One such case involves rigorous approximations of periodic solutions to nonlocal PDEs by neural networks, which generally lack periodic boundary conditions.

In this paper, we propose a new way to approximate these nonlocal periodic operators. Specifically, we systematically study how the periodic first-order fractional Laplacian and Riesz transform behave when applied to nonperiodic functions. This investigation lays the groundwork for a fresh approach to approximating solutions to equations with nonlocal features using neural networks. To illustrate the idea, we specifically consider the following setup: we denote the box $[-\pi, \pi]^2$ by $\TT^2$, which will also be identified with the two-dimensional torus in the periodic settings. We define the following operators 
\begin{equation}\label{def:tildelambda} 
\tilde{\l} \phi (x)
= P.V. \int_{\TT^2} (\phi(x) - \phi (x+y)) K(y) dy, 
\end{equation} and 
\be \label{def:tildeR}
\tilde{R} \phi (x)
= P.V. \int_{\TT^2} \phi (x+y) R^*(y) dy,
\ee
acting on possibly non-periodic functions, where P.V. means the Cauchy principle value, and the kernels $K$ and $R^*$ are defined by 
\begin{align}
&&K(y) &= \frac{2\Gamma\left(\frac{3}{2}\right)}{|\Gamma \left(-\frac{1}{2}\right) | \pi} \left(\fr{1}{|y|^3} + \sum\limits_{k \in \ZZ^2 \setminus \left\{0\right\}} \frac{1}{|y-2\pi k|^3} \right) \quad \text{and} \nonumber \\
 &&R^*(y)& = \frac{y}{2\pi |y|^3} + \sum\limits_{k \in \ZZ^2 \setminus \left\{0\right\}} \left(\frac{y + 2\pi k}{2\pi |y+2\pi k|^3} - \frac{ k}{ |2\pi k|^3} \right). \nonumber
\end{align}
The operators $\tilde{\Lambda}$ and $\tilde{R}$ respectively coincide with the periodic fractional Laplacian $\l := \sqrt{-\Delta}$ and Riesz transform $R := \na \l^{-1}$ for periodic functions $\phi$ (see \cite{constantin2015long}).

The periodic operators $R$ and $\l$ exhibit crucial analytical features when applied to periodic functions, including but not limited to integration by parts identities and boundedness properties on many functional spaces. Expectedly, these tools break down when the periodicity of the inputs is absent.  To this end, we study the behavior of $\tilde{\Lambda}$ and $\tilde{R}$ when acting on non-periodic functions and address the following questions:
\begin{enumerate}
    \item On which class of functions are $\tilde{\Lambda}$ and $\tilde{R}$ well-defined?
    \item How smooth are  $\tilde{\Lambda}$ and $\tilde{R}$?
    \item How are the inner products $\int_{\TT^2} \phi \tilde{\Lambda} \phi dx$ controlled?
    \item Is the operator $\tilde{R}$ bounded on $L^p$ spaces for $p \in (1, \infty)$?
\end{enumerate} 

We then apply the developed theory to the problem of approximating solutions to the periodic critical SQG equation by non-periodic physics-informed neural networks and rigorously estimate the resulting errors. The main challenges arise from seeking good control of the distances between the actual and approximating solutions. Due to the lack of periodicity of the neural networks, the time evolution of these distances is governed by unexpected nondissipative effects driven by the operator $\tilde{\l}$. However, we successfully derive a non-periodic maximum principle by which these distances are shown to be sufficiently small for all times. 

The idea of approximating nonlocal periodic operators by their non-periodic analogs constitutes the major novelty of this paper. Moreover, the techniques employed here appear promising for the development of a broader approximation theory for PDEs governed by operators with singular integral representations and different types of boundary conditions (e.g., Dirichlet and Neumann).

This paper is organized as follows. In section \ref{S2}, we introduce the functional spaces and notations that will be frequently used. In section \ref{S3}, we study the properties of the nonlocal non-periodic operators $\tilde \Lambda$ and $\tilde R$. Section \ref{S4} is dedicated to the construction of neural networks approximating the unique solution to the periodic critical SQG equation and the analysis of the resulting errors. The global smoothness is required herein and is established in Appendix~\ref{AP}.

\section{Functional Spaces} \la{S2}
 Let $X \subset \R^d$ where $d \in \left\{2,3\right\}$. For $1 \le p \le \infty$, we denote by $L^p(X)$ the Lebesgue spaces of measurable functions $f$ from $X$ to $\R$ (or $\RR^2)$ such that 
$$\|f\|_{L^p(X)} = \left(\int_{X} \|f\|^p dx\right)^{1/p} <\infty,$$ when $p \in [1, \infty)$ and
$$\|f\|_{L^{\infty}(X)} = esssup_{X}  |f| < \infty,$$ if $p = \infty$.
The $L^2$ inner product is denoted by $(\cdot,\cdot)_{L^2}$. For $k \in \NN$, we denote by $H^k(X)$ the classical Sobolev space of measurable functions $f$ from $X$ to $\R$ (or $\RR^2)$ 
such that  
$$ 
\|f\|_{H^k(X)}^2 = \sum\limits_{|\alpha| \le k} \|D^{\alpha}f\|_{L^2(X)}^2 < \infty.
$$ For $k \in \N$, we consider the spaces $C^k(X)$ of functions $f$ from $X$ to $\R$ (or $\RR^2$) such that 
$$\max\limits_{|\beta| \le k} \sup\limits_{x  \in X} |D^{\beta} f(x)| < \infty.$$ For a Banach space $(Y, \|\cdot\|_{Y})$ and $p\in [1,\infty]$, we consider the Lebesgue spaces $ L^p(0,T; Y)$ of functions $f$  from $Y$ to $\R$ (or $\RR^2)$ satisfying 
$$
\int_{0}^{T} \|f\|_{Y}^p dt  <\infty,
$$ with the usual convention when $p = \infty$. Finally, the letter $C$ will be used to denote a positive universal constant, which might change from line to line along the proofs. 

\section{Properties of the  Operators $\tilde \Lambda$ and $\tilde R$} \la{S3}

For $n \in \N$, let $n\TT^2=[-n\pi, n\pi]^2$. In this section, we investigate the properties of the operators $\tilde{\l}$ and $\tilde{R}$.

\begin{Thm} \la{trun11}
    The operator $\tilde{\l}$ defined in \eqref{def:tildelambda} satisfies the following properties: 
    \begin{enumerate}[ label=\textbf{(P\arabic*)}]
        \item\label{p1} Let $n \in \N$. Then $\tilde{\l}\phi (x)$ is defined for  $\phi \in C^2((n+1)\TT^2)$ and $x \in n\TT^2$.
        \item\label{p2} For $k\in\mathbb N$ and $\alpha \in \ZZ^2$ with $|\alpha| = k$, for any $\phi \in C^{k +2}( 2\TT^2)$, it holds that
        $
|D^{\alpha} \tilde{\l} \phi(x)| \\\le C \|\phi\|_{W^{k+2, \infty}(2\TT^2)}$ for any $x 
\in \TT^2$. Here $C$ is a positive universal constant independent of $\phi$.
        \item\label{p3} For any $\phi \in C^2(5\TT^2)$, it holds that 
        $
(\phi, \tilde{\l}\phi)_{L^2(\TT^2)} \ge \mathcal{A}(\phi)      
$  where
\be \label{firstbound}
|\mathcal{A}(\phi)|
\le C\|\phi\|_{L^2(5\TT^2)}^2 + C\left|\int_{\pa \TT^2} \int_{\TT^2} \frac{y \phi(x+y)^2 \cdot n}{|y|^3} dy d\sigma(x) \right|,
\ee for some positive universal constant $C$ independent of $\phi$. Here $n$ is the outward unit normal to $\pa \TT^2$ and $d\sigma(x)$ is the surface measure. In particular, if $\phi = \psi - \widehat{\psi}$ where $\psi$ is a periodic function on $\TT^2$ and $\widehat{\psi}$ is a smooth function on $2\TT^2$, it holds that 
\be \la{secondbound}
\beg{aligned}
&|\mathcal A(\psi - \widehat{\psi})| \le C\|\psi - \widehat{\psi}\|_{L^2(5\TT^2)}^2 
+ C\|\widehat{\psi}^2 (\cdot + 2\pi i) - \widehat{\psi}^2(\cdot)\|_{H^1(2\TT^2)} \\
&\quad + C\|\widehat{\psi}^2 (\cdot + 2\pi j) - \widehat{\psi}^2 (\cdot)\|_{H^1(2\TT^2)} + C\left(\|\psi\|_{L^{\infty}(\TT^2)} + \|\na \psi\|_{L^\infty(\TT^2)} \right) \\
& \quad \quad\quad\times \left(\|\widehat{\psi}(\cdot + 2\pi i) - \widehat{\psi}(\cdot)\|_{H^1(2\TT^2)} + \|\widehat{\psi}(\cdot + 2\pi j) - \widehat{\psi}(\cdot)\|_{H^1(2\TT^2)}\right).
\end{aligned}
\ee 
\end{enumerate}
\end{Thm}

\begin{proof} Below we denote by $c=\frac{2\Gamma(\frac32)}{\left|\Gamma(-\frac12) \right|\pi}$. For \ref{p1}, we fix $\epsilon > 0$ and estimate 
        \be \label{nonp1}
        \beg{aligned}
&\int_{|y| > \epsilon} \chi_{\TT^2}(y)(\phi(x) - \phi(x+y)) K(y) dy
\\= &c\int_{|y| > \epsilon} \chi_{\TT^2}(y) (\phi(x) - \phi(x+y)) \left(\sum\limits_{k \in \ZZ^2 \setminus \left\{0\right\}} \fr{1}{|y - 2\pi k|^3}\right) dy
\\&\quad+ c\int_{\epsilon < |y| \le \pi} \chi_{\TT^2}(y) \fr{\phi(x) - \phi(x+y)}{|y|^3} dy 
+ c\int_{|y| > \pi} \chi_{\TT^2}(y) \fr{\phi(x) - \phi(x+y)}{|y|^3} dy,
\end{aligned}
        \ee where $\chi_{\TT^2}$ denotes the characteristic function of the set $\TT^2$. 
    If $y \in \TT^2$, then 
    \begin{equation}\label{est:y-2pik}
        |y - 2\pi k| \ge 2\pi |k| - \pi \sqrt{2} \ge (2-\sqrt{2}) \pi |k|, 
    \end{equation}
    for any $k \in \ZZ^2 \setminus \left\{0\right\}$. Consequently, the first and last integrals on the right-hand side of \eqref{nonp1} can be bounded by a constant (independent of $\epsilon$) multiple of $\|\phi\|_{L^{\infty}((n+1)\TT^2)}$. Since $\fr{y}{|y|^3}$ is a Calderon-Zygmund kernel, it holds that $\int_{\epsilon < |y| \leq \pi } \frac{y}{|y|^3} dy = 0$. By making use of this latter fact and the mean-value theorem, we deduce the existence of a parameter $t \in (0,1)$ such that 
    \begin{equation*}
\beg{aligned}
\int_{\epsilon < |y| \le \pi} \chi_{\TT^2}(y) \fr{\phi(x) - \phi(x+y)}{|y|^3} dy
= \int_{\epsilon < |y| \le \pi }  \frac{y}{|y|^3} \cdot \left(\na \phi (x + (1-t)y) - \na \phi (x) \right) dy,
\end{aligned} 
\end{equation*}
which yields
\begin{equation*}
\left|\int_{\epsilon < |y| \le \pi} \chi_{\TT^2}(y) \fr{\phi(x) - \phi(x+y)}{|y|^3} dy\right| \le \|\na \na \phi\|_{L^{\infty}(n\TT^2)} \int_{ |y| \le \pi} \frac{1}{|y|} dy \le C\|\na \na \Phi\|_{L^{\infty}(n\TT^2)},
\end{equation*}
after a second application of the mean-value theorem, with a constant $C$ independent of $\epsilon$. Consequently, $\tilde{\l}\phi$ is well-defined for any $\phi \in C^2((n+1)\TT^2)$. 

\smallskip

\noindent {\bf \ref{p2}:} Fix $x 
\in \TT^2$. Using the commutativity of $\tilde{\l}$ and $D^{\alpha}$ and Property \ref{p1}, we have 
\begin{align*}
|D^{\alpha} \tilde{\l} \phi(x)| 
&= c\left|\int_{\TT^2} \left(D^{\alpha}\phi(x) - D^{\alpha}\phi(x+y) \right) K(y) dy\right| \\
&\le C\bigg(\|D^{\alpha}\phi\|_{L^{\infty}(2\TT^2)} +\sum\limits_{|\beta| = k +2} \|D^{\beta} \phi\|_{L^{\infty}(2\TT^2)}\bigg).
\end{align*}

\noindent {\bf \ref{p3}:} For $x \in \TT^2$ and $\phi \in C^2(5\TT^2)$, we decompose $\tilde{\l} \phi(x)$ into the sum
\begin{equation*} 
\tilde{\l}\phi (x)= c\int_{\TT^2} \frac{\phi(x) - \phi(x+y)}{|y|^3} dy 
+ c\int_{\TT^2} \sum\limits_{k 
\in \ZZ^2 \setminus \left\{0\right\}} \frac{\phi(x) - \phi(x+y)}{|y - 2\pi k|^3} dy
:= \tilde{\l}_1\phi (x) + \tilde{\l}_2 \phi (x).
\end{equation*}
For any $y \in \TT^2$ we have $|y - 2\pi k|^3 \ge C|k|^3$, and so
\be 
|(\phi, \tilde{\l}_2\phi)_{L^2(\TT^2)}|
\le C\int_{\TT^2} \int_{\TT^2} |\phi(x)| |\phi(x) - \phi(x+y)| dydx
\le C\|\phi\|_{L^2(2\TT^2)}^2. \nonumber
\ee 
In order to estimate the inner product $(\phi, \tilde{\l}_1 \phi)_{L^2(\TT^2)}$, we take a smooth cutoff function $\eta$ such that $\eta \equiv 1$ on $2\TT^2$ and $\eta \equiv 0$ outside $3\TT^2$. For $x, y \in \TT^2$, we have $\eta(x) = \eta(x+y) = 1$, hence we can rewrite $(\phi, \tilde{\l}_1 \phi)_{L^2(\TT^2)}$ as
\begin{equation*} 
\beg{aligned}
(\phi, \tilde{\l}_1 \phi)_{L^2(\TT^2)}
&= c\int_{\TT^2} \int_{\TT^2} \frac{\eta(x) \phi(x) (\eta(x) \phi(x) - \eta(x+y) \phi(x+y)
)}{|y|^3} dydx
\\&= c\int_{\TT^2} \int_{\RR^2} \frac{\eta(x) \phi(x) (\eta(x) \phi(x) - \eta(x+y) \phi(x+y)
)}{|y|^3} dydx
\\&\quad\quad- c\int_{\TT^2} \int_{\RR^2 \setminus \TT^2} \frac{\eta(x) \phi(x) (\eta(x) \phi(x) - \eta(x+y) \phi(x+y)
)}{|y|^3} dydx
\\&:=A_1 + A_2.
\end{aligned}
\end{equation*}
Since $\eta$ is supported on $3\TT^2$, we have $\eta(x) =1$ and $\eta(x+y) = 0$ for $x \in \TT^2$ and $y \in \RR^2 \setminus 4\TT^2$. Thus, we can estimate $A_2$ as follows,
\be 
|A_2| \le C\int_{\TT^2} \int_{4\TT^2 \setminus \TT^2} \left(\phi(x)^2 + |\phi(x)| |\phi(x+y)| \right) dydx 
\le C\|\phi\|_{L^2(5\TT^2)}^2. \nonumber
\ee As $A_1$ is a constant multiple of $(\eta \phi, \l_{\RR^2} (\eta \phi))_{L^2(\TT^2)}$ where $\l_{\RR^2} = \sqrt{-\Delta}$ is the fractional Laplacian on $\RR^2$, we can bound $A_1$ from below by 
\be 
A_1 = C\int_{\TT^2} \eta(x) \phi(x) \l_{\RR^2} (\eta \phi)(x) dx 
\ge C\int_{\TT^2} \l_{\RR^2} (\eta^2 \phi^2) (x) dx, \nonumber
\ee where the last inequality follows from the estimate $2\eta(x) \phi(x) \l_{\RR^2} (\eta \phi)(x) \ge \l_{\RR^2} (\eta^2 \phi^2) (x)$ that holds for a.e. $x \in \RR^2$ (see \cite[Proposition~2.3]{cordoba2004maximum}). As $\l_{\RR^2}^2 = - \Delta$, we have 
\begin{align*}
    A_1 &\ge C\int_{\TT^2} - \Delta \l_{\R^2}^{-1} (\eta^2 \phi^2)(x) dx 
= C \int_{\TT^2} - \na \cdot (\na \l_{\R^2}^{-1} (\eta^2 \phi^2))(x) dx \\
& = C \int_{\pa \TT^2} - R_{\RR^2} (\eta^2 \phi^2)(x) \cdot n d \sigma(x),
\end{align*} 
after integrating by parts and making use of the identity $\na \l_{\R^2}^{-1} = R_{\RR^2}$ where $R_{\RR^2}$ is the Riesz transform on $\RR^2$. For a fixed $x \in \pa \TT^2$, we can decompose $R_{\RR^2}(\eta^2 \phi^2)(x)$ as 
\begin{align*}
R_{\RR^2}(\eta^2 \phi^2)(x)
&= c\int_{\RR^2} \frac{y \eta(x+y)^2 \phi(x+y)^2}{|y|^3} dy \\
&= c\int_{\RR^2 \setminus \TT^2} \frac{y \eta(x+y)^2 \phi(x+y)^2}{|y|^3} dy 
+ c\int_{\TT^2} \frac{y \eta(x+y)^2 \phi(x+y)^2}{|y|^3} dy.
\end{align*}
Since $\eta(x+y) = 1$ for $x \in \pa \TT^2$ and $y \in \TT^2$, and $\eta(x+y) = 0$ for $x \in \pa \TT^2$ and $y \in \RR^2 \setminus 4\TT^2$, this latter decomposition boils down to 
\begin{align*}
R_{\RR^2}(\eta^2 \phi^2)(x)
&= c\int_{4\TT^2 \setminus \TT^2} \frac{y \eta(x+y)^2 \phi(x+y)^2}{|y|^3} dy 
+ c\int_{\TT^2} \frac{y \phi(x+y)^2}{|y|^3} dy \\
&:= R_1(\eta^2 \phi^2)(x) + R_2\phi^2(x).
\end{align*}
Obviously, $|R_1(\eta^2 \phi^2)(x)| \le C\|\phi\|_{L^2(5\TT^2)}^2$ for $x\in\partial \TT^2$, which implies that 
\be 
\left|\int_{\pa \TT^2} R_1(\eta^2\phi^2)(x) \cdot n d\sigma(x)\right| \le C\|\phi\|_{L^2(5\TT^2)}^2. \nonumber
\ee  
Combining above yields \eqref{firstbound}. Now suppose that $\phi = \psi - \widehat{\psi}$ where $\psi$ is periodic on $\TT^2$ and $\widehat{\psi}$ is smooth on $2\TT^2$. In this case, we have 
\be \la{decompo}
R_2(\psi - \widehat{\psi})^2(x) = R_2(\psi^2)(x) + R_2(\widehat{\psi}^2)(x) -2c\int_{\TT^2} \frac{y \psi(x+y)\widehat{\psi}(x+y)}{|y|^3} dy.
\ee Since $\psi$ is periodic, the following cancellation
\be 
\int_{\pa \TT^2} R_2(\psi^2)(x) \cdot n d\sigma(x) = 0 \nonumber
\ee holds. As for the two-dimensional operator $R_2 = (R_{2,1}, R_{2,2})$ applied to $\widehat{\psi}^2$, we have 
\begin{align*}
\int_{\pa \TT^2} R_2(\widehat{\psi}^2)(x) \cdot n d\sigma(x)
&= \int_{-\pi}^{\pi} \left[R_{2,1}(\widehat{\psi}^2) (\pi, x_2) - R_{2,1} (\widehat{\psi}^2)(-\pi, x_2) \right] dx_2 
\\&\quad\quad+ \int_{-\pi}^{\pi} \left[R_{2,2}(\widehat{\psi}^2)(x_1, \pi) - R_{2,2} (\widehat{\psi}^2)(x_1, - \pi) \right]dx_1.
\end{align*}
 By the trace theorem, the boundedness of $R_{2}$ from $L^2(\TT^2)$ into $L^2(2\TT^2)$ (see Theorem \ref{rien} below), and the fact that $R_2$ and $\na$ commutes, we have 
\beg{align*}
&\left|\int_{\pa \TT^2} R_2(\widehat{\psi}^2)(x)  \cdot n d\sigma(x)\right|\\
&\qquad \le C\|R_{2,1}(\widehat{\psi}^2) (\pi, \cdot) - R_{2,1} \widehat{\psi}^2(-\pi, \cdot)\|_{L^2(-\pi, \pi)} 
\\&\qquad\quad\quad+ C\|R_{2,2}(\widehat{\psi}^2) (\cdot, \pi) - R_{2,2} \widehat{\psi}^2(\cdot, -\pi)\|_{L^2(-\pi, \pi)}
\\&\qquad\le C\|R_{2,1} (\widehat{\psi}^2)(\cdot + 2\pi i) - R_{2,1} (\widehat{\psi}^2)(\cdot)\|_{H^1(\TT^2)}
\\&\qquad\quad\quad+ C\|R_{2,2} (\widehat{\psi}^2)(\cdot + 2\pi j) - R_{2,2} (\widehat{\psi}^2)(\cdot)\|_{H^1(\TT^2)}
\\&\qquad\le C\|\widehat{\psi}^2 (\cdot + 2\pi i) - \widehat{\psi}^2(\cdot)\|_{H^1(2\TT^2)}
+ C\|\widehat{\psi}^2 (\cdot + 2\pi j) - \widehat{\psi}^2 (\cdot)\|_{H^1(2\TT^2)}.
\end{align*}
 In order to obtain good control of the last term in \eqref{decompo}, we fix $x \in \pa \TT^2$ and estimate the differences \begin{small}
\begin{align*}
D_1 = \int_{-\pi}^{\pi} \int_{\TT^2} \left[\frac{y_1 \psi ((\pi, x_2) + y)\widehat{\psi}((\pi, x_2) + y)}{|y|^3} -\frac{y_1 \psi ((-\pi, x_2) + y)\widehat{\psi}((-\pi, x_2) + y)}{|y|^3} \right] dy dx_2,
\end{align*} 
\end{small}
and  \begin{small}
\begin{align*}
D_2 = \int_{-\pi}^{\pi} \int_{\TT^2} \left[\frac{y_2 \psi ((x_1, \pi) + y)\widehat{\psi}((x_1, \pi) + y)}{|y|^3} -\frac{y_2 \psi ((x_1, -\pi) + y)\widehat{\psi}((x_1, -\pi) + y)}{|y|^3} \right] dy dx_1.
\end{align*}
\end{small}Indeed, the periodicity of $\psi$ implies that $\psi((\pi, x_2) + y) - \psi ((-\pi, x_2) + y) = 0$, and consequently, we can rewrite $D_1$ as 
\beg{align*}
D_1 &= \int_{-\pi}^{\pi} \int_{\TT^2} \frac{y_1 \psi ((-\pi, x_2) + y) \left[\widehat{\psi} ((\pi, x_2) + y) - \widehat{\psi}((-\pi, x_2) + y) \right] }{|y|^3} dy dx_2 
\\&= \int_{-\pi}^{\pi} R_{2,1} \left[\psi(-\pi, \cdot) (\widehat{\psi}(\pi, \cdot) - \widehat{\psi}(-\pi, \cdot))\right]dx_2,
\end{align*}
 which can be bounded by
\be 
|D_1| \le C\|\psi(\cdot) \left[\widehat{\psi}(\cdot + 2\pi i) - \widehat{\psi}(\cdot) \right]\|_{H^1(2\TT^2)}, \nonumber
\ee after a straightforward application of the trace theorem and use of the boundedness of $R_{2}$ from $H^1(\TT^2)$ to $H^1(2\TT^2)$. In view of Sobolev product estimates, we infer that 
\be 
D_1 \le C\left(\|\psi\|_{L^{\infty}(2\TT^2)} + \|\na \psi\|_{L^\infty(2\TT^2)} \right) \|\widehat{\psi}(\cdot + 2\pi i) - \widehat{\psi}(\cdot) \|_{H^1(2\TT^2)}. \nonumber
\ee A similar approach yields 
\be 
D_2 \le C\left(\|\psi\|_{L^{\infty}(2\TT^2)} + \|\na \psi\|_{L^\infty(2\TT^2)} \right) \|\widehat{\psi}(\cdot + 2\pi j) - \widehat{\psi}(\cdot) \|_{H^1(2\TT^2)}. \nonumber
\ee 
As $\psi$ is periodic, $\|\psi\|_{L^{\infty}(2\TT^2)} = \|\psi\|_{L^{\infty}(\TT^2)}$ and $\|\na\psi\|_{L^{\infty}(2\TT^2)} = \|\na\psi\|_{L^{\infty}(\TT^2)}$.
Therefore, we conclude that \eqref{secondbound} holds. 
\end{proof}

Next, we address the properties of the operator $\tilde{R}$. 

\beg{Thm} \la{rien}  The operator $\tilde{R}$ defined in \eqref{def:tildeR} satisfies the following properties: 
\begin{enumerate}[ label=\textbf{(R\arabic*)}]
    \item\label{r1}Let $n \in \N$. Then $\tilde{R}\phi(x)$ is well-defined for any $\phi \in C^1((n+1)\TT^2)$ and $x \in n \TT^2$.
    \item\label{r2} For $|\alpha| = k$, it holds that
        $
|D^{\alpha} \tilde{R} \phi(x)|\le C \|\phi\|_{W^{k+1, \infty}(2\TT^2)}$ for any $\phi \in C^{k +1}(2\TT^2)$ and any $x \in \TT^2$. Here $C$ is a positive universal constant independent of $\phi$.
    \item\label{r3} It holds that $\|\tilde{R}\phi\|_{L^2(\TT^2)} \le C\|\phi\|_{L^2(2\TT^2)}$.
\end{enumerate}
\end{Thm} 

\begin{proof} We establish the properties \ref{r1}, \ref{r2} and \ref{r3} separately. 

\smallskip

\noindent {\bf \ref{r1}:} The vanishing of the spatial integral of the kernel $\frac{y}{|y|^3}$ over the annulus $\epsilon < |y|<\pi$ gives 
    \be 
\left|\int_{\epsilon < |y| <\pi} \frac{\phi(x+y)y}{|y|^3} dy\right|
= \left|\int_{\epsilon < |y| <\pi} \frac{(\phi(x+y) - \phi(x))y}{|y|^3} dy \right|
\le C\|\na \phi\|_{L^{\infty}((n+1)\TT^2)}, \nonumber
    \ee 
    where we have applied the mean-value theorem, and the constant $C$ is independent of $\epsilon$. 
    For a fixed $k\in \mathbb Z^2 \setminus \{0\}$, we consider the function $f(y) = \frac{y+2\pi k}{2\pi|y+2\pi k|^3}$. For $|y|>\epsilon$, there exists some $t\in(0,1)$ such that $f(y)-f(0)=\nabla f(ty) y$ by the mean-value theorem. A straightforward calculation and an application of \eqref{est:y-2pik} yield  $|\na f|\leq \frac{C}{|y+2\pi k|^3} \leq \frac{C}{|k|^3}$ for $|y|>\epsilon$ and $y\in\TT^2$, thus
    \beg{align*}
     &\left|\int_{|y| >\epsilon} \chi_{\mathbb T^2}(y) \phi(x+y) \sum\limits_{k \in \ZZ^2 \setminus \left\{0\right\}} \left(\frac{y + 2\pi k}{2\pi |y+2\pi k|^3} - \frac{2\pi k}{2\pi |2\pi k|^3} \right)  dy \right| 
    \\&\quad\quad\leq C \sum\limits_{k \in \ZZ^2 \setminus \left\{0\right\}}|k|^{-3} \|\phi\|_{L^\infty ((n+1)\TT^2)} \leq C\|\phi\|_{L^\infty ((n+1)\TT^2)}.
    \end{align*}

\noindent {\bf \ref{r2}:} The property \ref{r2} follows from \ref{r1} and the fact that $\tilde{R}$ commutes with differential operators.

\smallskip

\noindent {\bf \ref{r3}:} For a fixed $x \in \TT^2$, we have 
    \beg{align*}
\tilde{R}\phi(x) &= \int_{\TT^2} \phi(x+y) R^*(y) dy
\\&= \int_{\RR^2} \chi_{\TT^2 + \TT^2}(x+y) \phi(x+y) \chi_{\TT^2} (y)R^*(y) dy = \tilde{R} (\chi_{\TT^2 + \TT^2} \phi)(x),
    \end{align*} where the second equality follows from the fact that both $x$ and $y$ are in $\TT^2$. Thus, we deduce that $\|\tilde{R}(\chi_{\TT^2 + {\TT^2} }\phi)\|_{L^2(\RR^2)} \le C\|\chi_{\TT^2+\TT^2}\phi\|_{L^2(\R^2)}$ (see \cite[Chapter~II Theorem~1]{stein1970singular}). As $\tilde{R}(\chi_{\TT^2 + \TT^2} \phi)$ and $\tilde{R}\phi$ coincides on $\TT^2$, we obtain the bound 
    \beg{align*}
    \|\tilde{R} \phi\|_{L^2(\TT^2)} 
    &= \|\tilde{R}(\chi_{\TT^2 + {\TT^2} }\phi)\|_{L^2(\TT^2)}
    \le \|\tilde{R}(\chi_{\TT^2 + {\TT^2} }\phi)\|_{L^2(\RR^2)} 
    \\&\le C\|\phi\|_{L^2(\TT^2 + \TT^2)} \le C\|\phi\|_{L^2(2\TT^2)}.
    \end{align*}
\end{proof}

\section{Applications to the Critical SQG Equation} \la{S4}
In the section, we applied the developed theory for the operators $\tilde \l$ and $\tilde R$ to prove the existence of neural networks approximating the  solutions to the two-dimensional critically dissipative surface quasi-geostrophic (SQG) equation:
\be \label{SQG}
\pa_t \psi + u \cdot \na \psi + \l \psi = 0, \quad
 \text{with the advection velocity} \quad  
u = R^{\perp} \psi, 
\ee where $\l = \sqrt{-\Delta}$ is the fractional Laplacian of order 1 and $ $$R^{\perp} = (-R_2, R_1)$ is a rotation of the two-dimensional Riesz transform $R = (R_1, R_2) = \na \l^{-1}$.  The equation \eqref{SQG} is posed on the box  $\TT^2 = [-\pi, \pi]^2$ equipped with periodic boundary conditions. The initial data associated with \eqref{SQG}
is  assumed to have a vanishing spatial average over $\TT^2$ and is denoted by 
\be \label{initialdata}
\psi (x,0) = \psi_0(x), \quad \forall x \in \TT^2.
\ee 

The critical SQG equation describes the time evolution of a surface temperature in rapidly rotating stratified fluids and is interesting from a physical point of view due to its relevance in geophysics, meteorology, and oceanography \cite{constantin1994formation,held1995surface}. It is mathematically challenging due to its nonlocal and nonlinear structure by which the advection and dissipation have the same differential order, balancing rather than being dissipatedly dominated by each other. As a matter of fact, the global regularity of solutions to the critical SQG model has been a center of interest over the last three decades in the presence and absence of spatial boundaries (see \cite{caffarelli2010drift, constantin2016critical, constantin2020estimates, constantin2018local, constantin2015long, constantin2012nonlinear, constantin1999behavior, kiselev2007global, ignatova2019construction, stokols2020holder} and references therein). Its numerical solutions have also been explored in works such as
\cite{constantin1994singular, held1995surface, constantin2012new, ohkitani2012asymptotics, song2017fractional, bonito2021numerical}. 

Among numerical methods to solve PDEs, neural network approaches \cite{raissi2018hidden,raissi2019physics,sirignano2018dgm,HaJeE:18,beck2019machine,yu2018deep} have received great attention in recent years, due to their ability and flexibility to approximate complex and/or high-dimensional functions, as well as their ease of implementation. In this section, we construct  neural networks approximating the solutions to the periodic critical SQG equation in all Sobolev spaces, and we rigorously estimate the errors resulting from these approximations. This construction and analysis are not trivial as neural networks are generally non-periodic functions while the nonlocal operators (namely $\l$ and $R^{\perp}$) involved in the evolution of the temperature $\psi$ are defined on a class of regular periodic functions. In other words, neural networks, unless intentionally constructed for the purpose, typically do not belong to the domains of $\l$ and $R^{\perp}$. This gives rise to the necessity for the nonlocal non-periodic operators $\tilde{\l}$ and $\tilde{R}^{\perp}$ to suitably approximate $\l$ and $R^{\perp}$, respectively. 

\subsection{Physics-Informed Neural Networks (PINNs)} 
We consider neural networks $\psi_{\theta}$ defined by composition of layer functions of the form $\sigma (\omega x + b)$, where $\omega = (\omega_1, \dots, \omega_d)$ is a weight vector, $x \in \R^d$ is the input, and  $b \in \R$ is a bias parameter. The subscript $\theta$ denotes the collection of all the weights and biases. The function $\sigma$ is called an activation function and will be taken to be $\sigma = \tanh (\xi) = \frac{e^{\xi} - e^{-\xi}}{e^{\xi} + e^{-\xi}}$ in this paper. We refer to \cite{higham2019deep} for a more comprehensive mathematical introduction to network architectures.

A common approach to determine the optimal parameters $\theta^\ast$ for $\psi_\theta$ is to minimize the residuals arising from the PDE operators, initial conditions, and the boundary conditions, as described in the Physics-Informed Neural Networks (PINNs) method.
PINNs 
were initially introduced in the 90s by \cite{dissanayake1994neural,lagaris1998artificial,lagaris2000neural} and experienced a resurgence in recent years, spearheaded by \cite{raissi2018hidden,raissi2019physics}. Subsequently, there has been a remarkable surge in the development and utilization of PINNs for diverse applications related to various PDEs. An abbreviated compilation of references includes \cite{cuomo2022scientific,jagtap2021extended,jagtap2020conservative,lu2021deepxde,mao2020physics,karniadakis2021physics} and references therein. It is noteworthy that the deep Galerkin method \cite{sirignano2018dgm} shares a comparable conceptual foundation with PINNs.

In addition to evaluating the computational efficiency and accuracy of PINNs in numerically solving PDEs, researchers are actively engaged in a thorough study of their theoretical guarantees. An abbreviated list of works, including \cite{mishra2022estimates1, mishra2022estimates, de2022error, de2023error, de2022generic, biswas2022error}, showcases authors who have conducted comprehensive investigations into the error analysis of PINNs, examining their effectiveness in approximating various types of PDEs. However, we believe this is the first work that addresses the rigorous PINNs analysis of PDEs involving nonlocal periodic operators.

This section is devoted to the rigorous error analysis of PINNs for approximating solutions to critical SQG equations.
For that objective, we consider various types of residuals. The pointwise \emph{PDE residual} is defined as 
\begin{equation*}
\mathcal{R}_{i}[\theta] (t,x) = \pa_t \psi_{\theta} + \tilde{R}^{\perp}\psi_{\theta} \cdot \na \psi_{\theta} + \tilde{\l} \psi_{\theta}, \quad (t,x)\in [0,T]\times\mathbb T^2,
\end{equation*}
and the \emph{initial residual} is denoted by
\begin{equation*}
   \mathcal{R}_{t}[\theta] (x) = \psi_{\theta}(x,0) - \psi_0, \quad x\in \mathbb T^2.
\end{equation*}
For a regularity index $s\in\mathbb N$, define the \emph{boundary residual} as
\begin{align*}
\mathcal{R}&_{b}[s;\theta](t,x) = \mathcal{R}_{b}[s;\theta](t,x_1,x_2) \\
&= \left[\sum\limits_{|\alpha| \le s} \left(D^{\alpha} \psi_{\theta} (x_1, \pi) - D^{\alpha} \psi_{\theta}(x_1, -\pi)\right)^2 +\left(D^{\alpha} \psi_{\theta} (\pi, x_2) - D^{\alpha} \psi_{\theta}(-\pi, x_2)\right)^2 \right] \\&=:\mathcal{R}_{b,1}[s;\theta](t,x_1) + \mathcal{R}_{b,2}[s;\theta](t,x_2), \quad t\in[0,T],\, x_1,x_2\in [-\pi,\pi],
\end{align*}
where $x_1$ and $x_2$ are the first and the second component of $x$,
and the \emph{periodicity residual} as 
\begin{align*}
\mathcal{R}&_{per}[s;\theta] (t,x) \\
&= \sum\limits_{|\alpha| \le s} \sum\limits_{k,m=-2}^{2} (D^{\alpha} \psi_{\theta}(x) - D^{\alpha}\psi_{\theta} (x - 2k\pi i - 2m\pi j))^2
\\&\quad+  \sum\limits_{|\alpha| \le s+1} (D^{\alpha} \psi_{\theta}(x) - D^{\alpha}\psi_{\theta} (x + 2\pi i))^2 + (D^{\alpha} \psi_{\theta}(x) - D^{\alpha}\psi_{\theta} (x + 2\pi j))^2
\\&\quad\quad+ \sum\limits_{|\beta| \le 1} \sum\limits_{|\alpha| \le s} |D^{\beta} (D^{\alpha} \psi_{\theta})^2 (x+ 2 \pi i) - D^{\beta} (D^{\alpha} \psi_{\theta})^2(x) |^2
\\&\quad\quad\quad+\sum\limits_{|\beta| \le 1}  \sum\limits_{|\alpha| \le s} |D^{\beta} (D^{\alpha} \psi_{\theta})^2 (x+ 2 \pi j) - D^{\beta} (D^{\alpha} \psi_{\theta})^2(x) |^2, \quad (t,x)\in [0,T]\times 2\mathbb T^2,
\end{align*}
where $i= (1,0)$ and $j = (0,1)$. 
The optimal parameters, denoted as $\theta^\ast$, are sought to minimize all these residuals simultaneously. Specifically, when $\psi_{\theta^\ast} = \psi$ represents the true periodic solution, all residuals are expected to vanish. It is important to emphasize that the introduction of the periodicity residual $\mathcal{R}_{per}$ distinguishes our work from prior studies on PINNs. As will be demonstrated later, this addition is essential for controlling the correctors arising from the estimates of the non-periodic operators $\tilde{\l}$ and $\tilde{R}$.

To analyze the performance of neural networks approximating the solution to the SQG equation, we categorize errors into two types: the \emph{generalization error} and the  \emph{total error}.
For a regularity index $s\in\mathbb N$,
the generalization error is defined to be 
\be \label{generaler}
\mathcal{E}_{G}[s;\theta] = \left(\mathcal{E}_{G}^i [s;\theta]^2 + \mathcal{E}_{G}^t [s;\theta]^2 + \mathcal{E}_{G}^b [s;\theta]^2 + \mathcal{E}_{G}^{per} [s;\theta]^2 + \lambda \mathcal{E}_{G}^p [s;\theta]^2 \right)^{\fr{1}{2}},
\ee 
where 
\begin{align*}
   & \mathcal{E}_{G}^i [s;\theta]^2 
= \int_{0}^{T} \|\mathcal{R}_{i}[\theta](t)\|_{H^s(\TT^2)}^2 dt,
\\
&\mathcal{E}_{G}^t [s;\theta]^2
= \|\mathcal{R}_{t}[\theta]\|_{H^s(\TT^2)}^2,
\\
&\mathcal{E}_{G}^b [s;\theta]^2
=  \int_0^T \left(\int_{-\pi}^\pi \mathcal{R}_{b,1}[s;\theta](t,x_1) dx_1 + \int_{-\pi}^\pi \mathcal{R}_{b,2}[s;\theta](t,x_2) dx_2\right) dt ,
\\
&\mathcal{E}_G^{per}[s;\theta]^2 = \int_{0}^{T}  \int_{2\TT^2} \mathcal{R}_{per}[s;\theta](t) dt,
\\
&\mathcal{E}_{G}^p [s;\theta]^2
= \int_{0}^{T} \|\psi_{\theta}(t) \|_{H^{s+3} (\TT^2)}^2 dt.
\end{align*}
 Here, $\lambda$ is a constant to be determined later. The term $\mathcal{E}_{G}^p$ represents the penalty term and is inspired from \cite{biswas2022error}. It is crucial in order to obtain \textit{a priori} error estimate independent of the neural networks. 

The total error, defined by 
\be 
\mathcal{E}[s; \theta]^2 = \int_{0}^{T} \|\psi - \psi_{\theta}\|_{H^s(\TT^2)}^2 dt, \nonumber
\ee 
measures the distance between the neural network approximation and the true solution to \eqref{SQG}--\eqref{initialdata}. In the sequel, we aim to answer the following two questions:

\begin{enumerate}[ label=\textbf{Q\arabic*}]
    \item\label{Q1}\hspace{-4pt}: The existence of neural networks $\psi_\theta$ such that the generalization error $\mathcal{E}_{G}[s;\theta] < \epsilon$ for arbitrary $\epsilon > 0$ (\emph{cf.} {\bf Theorem \ref{generalmain}}); 
    \item\label{SubQ2}\hspace{-4pt}: The control of the total error by the generalization error, i.e., $\mathcal{E}[s; \theta] \lesssim \mathcal{E}_{G}[s;\theta]$ (\emph{cf.} {\bf Theorem \ref{Q2answer}});
  
\end{enumerate}

\subsection{Error Estimates} In this subsection, we answer the two questions raised above, and we need the following lemma:

\beg{lem}[{\cite[Lemma A.1]{hu2023higher}}] Let $\Omega = \prod\limits_{i=1}^{d} [a_i, b_i]$. Suppose $f \in H^m(\Omega).$ Let $N>5$ be an integer. Then there exists a \texttt{tanh} neural network $\widehat{f}^N$ such that 
$$\|f - \widehat{f}^N\|_{H^k(\Omega)} \le C_{k,d,f,\Omega} (1 + \ln^k N)N^{-m+k}, $$ for any $k \in \left\{0, 1, \dots, m-1 \right\}.$
\end{lem}

\beg{Thm} \label{generalmain} Let $s \ge 7$ and $n \ge 2$ be integers. Let $T>0$ be an arbitrary time. Suppose that $\psi_0 \in H^s(\TT^2)$. Fix a small parameter $\epsilon >0$ and an integer $\ell \in \left\{0, \dots, s-7\right\}$. There exists a real number $\lambda >0$, an integer $N>5$ depending on $\epsilon$ and $s$, and a \texttt{tanh} neural network $\widehat{\psi}= \psi_{\theta}^N$ 
such that $\mathcal{E}_G[\ell;\theta] \le \epsilon.$ 
\end{Thm}

\begin{proof} Since the initial data $\psi_0 \in H^s(\TT^2)$, it follows from Theorem \ref{smoothnessSQG} that the exact solution $\psi$ to the SQG problem \eqref{SQG}--\eqref{initialdata} is $C^{s-2}([0,T]\times \TT^2)$. As $\psi$ is periodic, it also follows that $\psi\in C^{s-2}([0,T]\times 6\TT^2)$. Consequently, there exists a neural network $\widehat{\psi}$ such that 
\be \label{neuraldecay}
\|\psi - \widehat{\psi}\|_{H^k([0,T] \times 6\TT^2)} \le C_{s, k,\psi,T} \frac{1+ \ln^k N}{N^{{ s-2-k}}},
\ee for any $k \in \left\{0,1, \dots, s-3\right\}$. Here, we abused notation and denoted the periodic extension of $\psi$ by $\psi$ as well. In order to estimate the generalization error, we implement the following six steps.

\smallskip
\noindent{\bf{Step 1. Estimation of $\mathcal{E}_{G}^i [\ell;\theta]$}.} We estimate the time derivative terms as follows,
\be 
\|\pa_t \psi - \pa_t \widehat{\psi}\|_{H^{\ell}([0,T] \times \TT^2)}  \le \|\psi - \widehat{\psi}\|_{H^{\ell + 1}([0,T] \times \TT^2)} 
\le C\frac{1 + \ln^{\ell + 1}N}{N^{-\ell + s - 3}}, \nonumber
\ee for any $\ell \in \left\{0, 1, \dots, s-4 \right\}$. Using the fact that $\l$ and $\tilde{\l}$ coincide on the space of periodic functions, the property \ref{p2} of the operator $\tilde{\l}$, and standard continuous Sobolev embedding, we have
{\small\begin{equation*} 
\|\l\psi - \tilde{\l} \widehat{\psi}\|_{H^{\ell}([0,T] \times \TT^2)}
= \|\tilde{\l}(\psi - \widehat{\psi})\|_{H^{\ell}([0,T] \times \TT^2)}
\le C\|\psi - \widehat{\psi}\|_{H^{\ell + 4} ([0,T] \times 2\TT^2)}
\le C\frac{1 + \ln^{\ell + 4}N}{N^{-\ell + s - 6}},
\end{equation*}}
for any $\ell \in \left\{0, 1, \dots, s-7 \right\}$.
As $R$ and $\tilde{R}$ coincide for periodic functions, we have 
\be \nonumber
\beg{aligned}
&\|R^{\perp} \psi \cdot \na \psi - \tilde{R}^{\perp} \widehat{\psi} \cdot \na \widehat{\psi}\|_{H^{\ell}([0,T] \times \TT^2)} 
\\&\quad\quad\le \|\tilde{R}^{\perp} (\psi - \widehat{\psi}) \cdot \na \widehat{\psi} \|_{H^{\ell}([0,T] \times \TT^2)} 
+ \|R^{\perp} {\psi} \cdot \na (\psi - \widehat{\psi})\|_{H^{\ell}([0,T] \times \TT^2)}
\\&\quad\quad\le \|\tilde{R}^{\perp} (\psi - \widehat{\psi}) \cdot \na \widehat{\psi} \|_{H^{\ell+1}([0,T] \times \TT^2)} 
+ \|R^{\perp} {\psi} (\psi - \widehat{\psi})\|_{H^{\ell+1}([0,T] \times \TT^2)}
\\&\quad\quad\le C\|\tilde{R}^{\perp}(\psi - \widehat{\psi})\|_{H^{\ell +1}([0,T] \times \TT^2)} \|\widehat{\psi}\|_{H^{\ell +2}([0,T] \times \TT^2)}  
\\&\quad\quad\quad\quad\quad\quad+ C\|R^{\perp}{\psi}\|_{H^{\ell +1}([0,T] \times \TT^2)}  \|\psi - \widehat{\psi}\|_{H^{\ell +1}([0,T] \times \TT^2)}, 
\end{aligned}
\ee for any $\ell \ge 1$, where the divergence-free condition of the Riesz transform $R^{\perp}$ and the Banach Algebra property of the space-time Sobolev spaces are exploited. As $R^{\perp}$ is bounded spatially on $H^{\ell +1}(\TT^2)$, and $\pa_t$ and $R^{\perp}$ commutes, we infer that $R^{\perp}$ is bounded on $H^{\ell +1}([0,T] \times \TT^2)$. Moreover, the non-periodic operator $\tilde{R}$ maps $H^{\ell + 4}([0,T] \times 2\TT^2)$ continuously to $H^{\ell+1 }([0,T] \times \TT^2)$, a fact that follows from the property \ref{r2} and continuous Sobolev embedding. Consequently, we obtain the bound 
\be \nonumber
\beg{aligned} 
&\|R^{\perp} \psi \cdot \na \psi - \tilde{R}^{\perp} \widehat{\psi} \cdot \na \widehat{\psi}\|_{H^{\ell}([0,T] \times \TT^2)} 
\\&\quad\quad\le C\left(\|\psi\|_{H^{\ell+1}([0,T] \times \TT^2)}  + \|\widehat{\psi}\|_{H^{\ell+2}([0,T] \times \TT^2)}  \right) \|\psi - \widehat{\psi}\|_{H^{\ell + 4}([0,T] \times 2\TT^2)}.
\end{aligned}
\ee Using the inequality
\be 
\|\widehat{\psi}\|_{H^{\ell+2}([0,T] \times \TT^2)}
\le \|\psi - \widehat{\psi}\|_{H^{\ell+2}([0,T] \times \TT^2)} + \|\psi\|_{H^{\ell+2}([0,T] \times \TT^2)}, \nonumber
\ee  we deduce that
\be \nonumber
\beg{aligned}
&\|R^{\perp} \psi \cdot \na \psi - \tilde{R}^{\perp} \widehat{\psi} \cdot \na \widehat{\psi}\|_{H^{\ell}([0,T] \times \TT^2)} 
\\&\quad\quad\le C\left(\|\psi\|_{H^{\ell+2}([0,T] \times \TT^2)}  + \|\psi - \widehat{\psi}\|_{H^{\ell+2}([0,T] \times \TT^2)}  \right) \|\psi - \widehat{\psi}\|_{H^{\ell + 4}([0,T] \times 2\TT^2)},
\end{aligned}
\ee
which yields
{\small\begin{equation*}
\|R^{\perp} \psi \cdot \na \psi - \tilde{R}^{\perp} \widehat{\psi} \cdot \na \widehat{\psi}\|_{H^{\ell}([0,T] \times \TT^2)} 
\le C\left(\|\psi\|_{H^{\ell+2}([0,T] \times \TT^2)}  + \frac{1+ \ln^{\ell +2}N}{N^{-\ell +s - 4}}\right) \frac{1+ \ln^{\ell +4}N}{N^{-\ell +s -6}},
\end{equation*}}
for all $\ell \in \left\{1, \dots, s-7  \right\}$ if $s\geq 8$, after making use of \eqref{neuraldecay}. Therefore, we conclude that
\be  \nonumber
\mathcal{E}_{G}^{i}[\ell; \theta]
\le C_{s,\ell,\psi,T} \left(1 + \|\psi\|_{H^{\ell + 2}([0,T] \times \TT^2)}  + \frac{1+ \ln^{\ell +2}N}{N^{-\ell +s - 4}} \right) \frac{1+ \ln^{\ell +4}N}{N^{-\ell +s - 6}},
\ee for any positive integer $\ell \le s-7$ when $\geq 8$. If $s=7$ and $\ell = 0$, then the nonlinear term can be bounded as follows,
\be \nonumber
\beg{aligned}
&\|R^{\perp} \psi \cdot \na \psi - \tilde{R}^{\perp} \widehat{\psi} \cdot \na \widehat{\psi}\|_{L^2([0,T] \times \TT^2)} 
\\&\quad\quad\le C\|\tilde{R}^{\perp}(\psi - \widehat{\psi})\|_{L^2([0,T] \times \TT^2)} \|\na \widehat{\psi}\|_{L^{\infty}([0,T] \times \TT^2)}  
\\&\quad\quad\quad\quad+ C\|R^{\perp}{\psi}\|_{L^{\infty}([0,T] \times \TT^2)}  \|\na (\psi - \widehat{\psi})\|_{L^2([0,T] \times \TT^2)} 
\\&\quad\quad\le C\|\psi - \widehat{\psi}\|_{H^3([0,T] \times 2\TT^2)} \left(\|\widehat{\psi}\|_{H^3([0,T] \times \TT^2)}  + \|\psi\|_{H^3([0,T] \times \TT^2)} \right)
\\&\quad\quad\le  C\|\psi - \widehat{\psi}\|_{H^3([0,T] \times 2\TT^2)} \left(\|\psi - \widehat{\psi}\|_{H^3([0,T] \times \TT^2)}  + \|\psi\|_{H^3([0,T] \times \TT^2)} \right).
\end{aligned}
\ee Thus we obtain
\be \nonumber
\mathcal{E}_{G}^{i}[0; \theta]
\le C_{s,\psi,T} \left(1 + \|\psi\|_{H^{3}([0,T] \times \TT^2)}  + \frac{1+ \ln^{3}N}{N^{s - 5}} \right) {\frac{1+ \ln^{3}N}{N^{s-5}}}.
\ee

\noindent{\bf{Step 2. Estimation of $\mathcal{E}_{G}^t[\ell;\theta]$}.} We denote by $\Omega$ the space-time box $[0,T] \times \TT^2$. By the trace theorem, we have
\be \nonumber
\beg{aligned}
\mathcal{E}_{G}^t[\ell;\theta]
&= \|\widehat{\psi}(x,0) - \psi(x,0)\|_{H^{\ell}(\TT^2)}
\le \sum\limits_{|\alpha| \le \ell }\|D^{\alpha}(\widehat{\psi} - \psi)\|_{L^2(\pa \Omega)}
\\&\le  C\sum\limits_{|\alpha| \le \ell} \|D^{\alpha}(\widehat{\psi} - \psi)\|_{H^1(\Omega)}
\le C\|\widehat{\psi} - \psi\|_{H^{\ell +1 }(\Omega)},
\end{aligned}
\ee from which we obtain the error estimate 
\be \nonumber
\mathcal{E}_{G}^t[\ell;\theta]
\le C\frac{1 + \ln^{\ell + 1}N}{N^{-\ell +s -3}},
\ee for any nonnegative integer $\ell \le s-4$, due to \eqref{neuraldecay}.

\smallskip

\noindent{\bf{Step 3. Estimation of $\mathcal{E}_{G}^b [\ell;\theta]$}.} Another application of the trace theorem and the periodic boundary conditions obeyed by the derivatives of the actual solution $\psi$ to \eqref{SQG}--\eqref{initialdata} yields 
\beg{align*}
&\mathcal{E}_{G}^b [\ell;\theta]^2
= \int_{0}^{T} \sum\limits_{|\alpha| \le \ell} \int_{-\pi}^{\pi}  \left(D^{\alpha} \widehat{\psi}(x_1,\pi) - D^{\alpha} \widehat{\psi}(x_1, -\pi) \right)^2 dx_1dt
\\&\quad\quad+ \int_{0}^{T}\sum\limits_{|\alpha| \le \ell} \int_{-\pi}^{\pi} \left(D^{\alpha} \widehat{\psi}(\pi,x_2) - D^{\alpha} \widehat{\psi}(-\pi, x_2) \right)^2 dx_2dt
\\&\le C\int_{0}^{T} \int_{-\pi}^{\pi} \sum\limits_{|\alpha| \le \ell} \left(\left(D^{\alpha} \widehat{\psi}(x_1,\pi) - D^{\alpha} \psi(x_1, \pi) \right)^2 + \left(D^{\alpha} \widehat{\psi}(x_1,-\pi) - D^{\alpha} \psi(x_1, -\pi) \right)^2\right) dx_1dt
\\&\quad\quad+C  \int_{0}^{T} \int_{-\pi}^{\pi} \sum\limits_{|\alpha| \le \ell} \left(\left(D^{\alpha} \widehat{\psi}(\pi,x_2) - D^{\alpha} \psi(\pi, x_2) \right)^2 +\left(D^{\alpha} \widehat{\psi}(-\pi,x_2) - D^{\alpha} \psi(-\pi, x_2) \right)^2 \right)dx_2dt
\\&\le C\sum\limits_{|\alpha| \le \ell} \|D^{\alpha} \widehat{\psi} - D^{\alpha} \psi\|_{L^2(\pa \Omega)}^2
\le C\sum\limits_{|\alpha| \le \ell} \|D^{\alpha} (\widehat{\psi} - \psi)\|_{H^1(\Omega)}^2
\le C\|\widehat{\psi} - \psi\|_{H^{\ell + 1}(\Omega)}^2,
\end{align*} 
and consequently, it holds that
\be \nonumber
\mathcal{E}_{G}^p [\ell;\theta]
\le C\frac{1 + \ln^{\ell + 1}N}{N^{-\ell +s -3}},
\ee for any nonnegative integer $\ell \le s-4$.

\smallskip

\noindent{\bf{Step 4. Estimation of $\mathcal{E}_{G}^p [\ell;\theta]$}.} Recall that we have fixed $\beta=3$ in $\mathcal{E}_{G}^{p}[\ell;\theta]$. By subtracting and adding $\psi$, we bound 
\be \nonumber
\mathcal{E}_{G}^p [\ell;\theta] ^2
\le C\int_{0}^{T} \|\widehat{\psi} - \psi\|_{H^{\ell +3}(\TT^2)}^2 dt
+ C\int_{0}^{T} \|\psi\|_{H^{\ell +3}(\TT^2)}^2dt.
\ee Thus, we obtain 
\be \nonumber
\mathcal{E}_{G}^p [\ell;\theta] \le C\frac{1+ \ln^{\ell +3}N}{N^{-\ell +s -5}} + C\|\psi\|_{C^{\ell +3}([0,T] \times \TT^2)},
\ee for any nonnegative integer $\ell \le s-6$. 

\smallskip

\noindent{\bf{Step 5. Estimation of $\mathcal{E}_{G}^{per}[\ell;\theta]$}.}  Using the periodicity of the exact solution $\psi$ to \eqref{SQG}--\eqref{initialdata}, we can estimate the error $\mathcal{E}_{G}^{per}[\ell; \theta]$ as 
\be \la{ccc}
\beg{aligned}
\mathcal{E}_{G}^{per}[\ell; \theta]^2
&\le  C\|\widehat{\psi} - \psi\|_{H^{\ell+1}([0,T] \times 6\TT^2)}^2 
\\&\quad+ \sum\limits_{|\beta| \le 1} \sum\limits_{|\alpha| \le \ell} \int_{2\TT^2} |D^{\beta} (D^{\alpha} \widehat{\psi})^2 (x+ 2 \pi i) - D^{\beta} (D^{\alpha} \widehat{\psi})^2(x) |^2 dx
\\&\quad\quad+\sum\limits_{|\beta| \le 1}  \sum\limits_{|\alpha| \le \ell} \int_{2\TT^2} |D^{\beta} (D^{\alpha} \widehat{\psi})^2 (x+ 2 \pi j) - D^{\beta} (D^{\alpha} \widehat{\psi})^2(x) |^2 dx.
\end{aligned}
\ee 
Due to the periodicity of $\psi$, we have $D^{\beta} (D^{\alpha} \psi)^2 (x+ 2\pi i) = D^{\beta} (D^{\alpha} \psi)^2 (x)$ and \\$D^{\beta} (D^{\alpha} \psi)^2 (x+ 2\pi j) = D^{\beta} (D^{\alpha}\psi)^2 (x)$. By subtracting and adding these terms, we can control the sums in \eqref{ccc} by 
\be \nonumber
\beg{aligned}
&C\sum\limits_{|\alpha| \le \ell} \|(D^{\alpha} \psi)^2  - (D^{\alpha} \widehat{\psi})^2\|_{H^1([0,T] \times 4\TT^2)}^2 
\\&= C\sum\limits_{|\alpha| \le \ell} \|(D^{\alpha} \psi  - D^{\alpha} \widehat{\psi}) (D^{\alpha} \psi  + D^{\alpha} \widehat{\psi})  \|_{H^2([0,T] \times 4\TT^2)}^2
\\&\le C\left(\|\psi\|_{H^{\ell + 2} ([0,T] \times 4\TT^2)}^2 + \|\widehat{\psi}\|_{H^{\ell + 2}([0,T] \times 4\TT^2)}^2 \right) \|\psi - \widehat{\psi}\|_{H^{\ell + 2} ([0,T] \times 4\TT^2)}^2
\\&\le C\left(\|\psi\|_{H^{\ell + 2} ([0,T] \times 4\TT^2)}^2 + \|\psi - \widehat{\psi}\|_{H^{\ell + 2}([0,T] \times 4\TT^2)}^2 \right) \|\psi - \widehat{\psi}\|_{H^{\ell + 2} ([0,T] \times 4\TT^2)}^2.
\end{aligned}
\ee 
Therefore,
\be \nonumber
\mathcal{E}_{G}^p [\ell;\theta] \le C\left(1 + \|\psi\|_{H^{\ell + 2} ([0,T] \times 4\TT^2)}^2 + \frac{1+ \ln^{\ell+2}N}{N^{-\ell +s -4}}\right)\frac{1+ \ln^{\ell+2}N}{N^{-\ell +s -4}},
\ee for any nonnegative integer $\ell \le s-5$.

\smallskip

\noindent{\bf{Step 6. Estimation of $\mathcal{E}_{G}[\ell;\theta]$}.} Combining the bounds derived in Steps 1--5, we deduce that 
\be 
\mathcal{E}_G[\ell;\theta] \le C\lambda^{\fr{1}{2}}\|\psi\|_{C^{\ell +3}([0,T] \times \TT^2)}  + C_{s,\psi,T,\ell} \frac{1+ \ln^{\ell +4}N}{N^{-\ell +s -6}}
+ C_{s,\psi,T,\ell} \left(\frac{1+ \ln^{\ell +4}N}{N^{-\ell +s -6}}\right)^2,
\ee
for any $\ell \in \left\{0, 1,\dots, s-7\right\}$. 
Choosing $\lambda$ sufficiently small and $N$ sufficiently large such that 
$$C\lambda^{\fr{1}{2}}\|\psi\|_{C^{\ell +3}([0,T] \times \TT^2)}   = \fr{\epsilon}{2} \mathrm{\;\;and\;\;} C_{s,\psi,T,\ell} \frac{1+ \ln^{\ell +4}N}{N^{-\ell +s -6}}
+ C_{s,\psi,T,\ell} \left(\frac{1+ \ln^{\ell +4}N}{N^{-\ell +s -6}}\right)^2
\le \fr{\epsilon}{2},$$ 
we infer that for these choices of $\lambda$ and $N$, the generalization error $\mathcal{E}_G[\ell;\theta] \le \epsilon.$
\end{proof}

We next provide estimates for the total error $\mathcal{E}$. However, we need first the following remark. 

\beg{rem} \la{remarkdissi} The property \ref{p3} is crucial for estimating the total error $\mathcal{E}$, as fractional integration by parts breaks down when applied to functions with non-periodic boundary conditions yielding undesired boundary terms.. Indeed, if $\psi$ is the exact periodic solution of the SQG problem and $\widehat{\psi}$ is a neural network approximating $\psi$, it follows from property \ref{p3} that 
\be 
(\psi - \widehat{\psi}, \tilde{\l} (\psi - \widehat{\psi}))_{L^2(\TT^2)} \ge \mathcal{A}(\psi - \widehat{\psi}), \nonumber
\ee where $\mathcal{A}(\psi - \widehat{\psi})$ obeys \eqref{secondbound}. 
 Due to the periodicity of the solution $\psi$, the first term on the right-hand side of \eqref{secondbound} can be bounded as follows, 
 \begin{equation*}
\begin{aligned} 
&\|\psi -\widehat{\psi}\|_{L^2(5\TT^2)}^2  
= \int_{5\TT^2} \|\psi(x) - \widehat{\psi}(x)|^2 dx
= \sum\limits_{k,m = -2}^{2} \int_{\TT^2 + 2k\pi i + 2m\pi j} |\psi(x) - \widehat{\psi}(x)|^2 dx\\
&= \sum\limits_{k,m=-2}^{2} \int_{\TT^2} |\psi (x - 2k\pi i - 2m\pi j) - \widehat{\psi}(x - 2k\pi i - 2m\pi j) |^2 dx  \\
&= \sum\limits_{k,m=-2}^{2} \int_{\TT^2} |\psi (x) - \widehat{\psi}(x - 2k\pi i - 2m\pi j) |^2 dx   \\
&\le C \int_{\TT^2} |\psi(x) - \widehat{\psi}(x)|^2 dx
+ C\sum\limits_{k,m=-2}^{2} \int_{\TT^2} |\widehat{\psi}(x) - \widehat{\psi}(x - 2k\pi i - 2m\pi j)|^2 dx. 
\end{aligned}
 \end{equation*}
The last term in the above estimate, together with the remaining terms in \eqref{secondbound} can be controlled by the $L^1(2\TT^2)$ norm of the periodic residual and its square root.  This will be exploited to show that the total error is fully dominated by the generalization error. 
\end{rem}

\beg{Thm} \la{Q2answer} Let $T>0$ be an arbitrary positive time and $s \ge 0$ be a nonnegative integer. Suppose $\psi_0 \in H^{s+3}(\TT^2)$. Let $\psi$ be the unique solution to the problem \eqref{SQG}--\eqref{initialdata}, and $\widehat{\psi}$ be \texttt{tanh} neural networks approximating $\psi$. Then there exists a positive constant $C_{T,0}$ depending only on $T$, $s$, and the initial size of $\psi_0$ such that the total error obeys 
\be \label{total3r}
\mathcal{E}[s;\theta]^2  
\le C_{T,0} \mathcal{E}_{G}[s;\theta]^2 \left(1+ \frac{1}{\sqrt{\lambda}} \right) e^{C_{T,0} + \frac{\mathcal{E}_G[s;\theta]}{\sqrt{\lambda}} },
\ee where $\lambda$ is the constant in \eqref{generaler}.
\end{Thm}

\begin{proof}
    The difference $\psi - \widehat{\psi}$ evolves according to 
    \be \nonumber
\pa_t (\psi - \widehat{\psi}) + \tilde{\l} (\psi - \widehat{\psi}) 
= - R^{\perp} \psi \cdot \na (\psi - \widehat{\psi})
- \tilde{R}^{\perp}(\psi - \widehat{\psi}) \cdot \na \widehat{\psi} 
- \mathcal{R}_{i}.
    \ee
We fix a vector $\alpha \in \NN^2$ such that $|\alpha| \le s$, apply the differential operator $D^{\alpha}$ to this latter equation, multiply by $D^{\alpha}(\psi -\widehat{\psi})$, integrate spatially over $\TT^2$, and sum over all such indices $\alpha$. We obtain the energy evolution 
\be \nonumber
\beg{aligned}
&\fr{1}{2} \fr{d}{dt} \|\psi - \widehat{\psi}\|_{H^s(\TT^2)}^2
+ \sum\limits_{|\alpha| \le s} (D^{\alpha} (\psi - \widehat{\psi}), \tilde{\l} D^{\alpha} (\psi - \widehat{\psi}))_{L^2(\TT^2)}
\\&\quad\quad= - \sum\limits_{|\alpha| \le s} (D^{\alpha} (R^{\perp} \psi \cdot \na (\psi - \widehat{\psi})) + D^{\alpha} (\tilde{R}^{\perp}(\psi - \widehat{\psi}) \cdot \na \widehat{\psi}), D^{\alpha}(\psi - \widehat{\psi}))_{L^2(\TT^2)} 
\\&\quad\quad\quad\quad
- \sum\limits_{|\alpha| \le s}(D^{\alpha} \mathcal{R}_{i}, D^{\alpha}(\psi - \widehat{\psi}))_{L^2(\TT^2)} .
\end{aligned}
\ee 
In view of Remark \ref{remarkdissi}, we bound the dissipation from below by $\mathcal{A}(D^{\alpha}(\psi - \widehat{\psi}))$, with 
\be \nonumber
\beg{aligned}
\sum\limits_{|\alpha| \le s}|\mathcal{A}(D^{\alpha}(\psi - \widehat{\psi})| 
&\le C\|\psi - \widehat{\psi}\|_{H^s(\TT^2)}^2 + C\left(1 + \|\psi\|_{W^{s+1,\infty}(\TT^2)}\right)\|\mathcal{R}_{per}[s]\|_{L^1(2\TT^2)}^{\fr{1}{2}}  \\&\quad\quad+C\|\mathcal{R}_{per}[s]\|_{L^1(2\TT^2)}.
\end{aligned}
\ee
Now we estimate the nonlinear terms. Since $R^{\perp} \psi$ is divergence-free, it holds that 
\be \nonumber
\beg{aligned}
&\left|\sum\limits_{|\alpha| \le s}(R^{\perp} \psi \cdot \na D^{\alpha} (\psi - \widehat{\psi}), D^{\alpha}(\psi - \widehat{\psi}))_{L^2(\TT^2)} \right|
\\&\quad\quad= \left|\sum\limits_{|\alpha| \le s} \int_{\pa \TT^2} R^{\perp} \psi  |D^{\alpha}(\psi - \widehat{\psi})|^2 \cdot n d\sigma(x) \right|
\\&\quad\quad\le C\|R^{\perp} \psi\|_{L^{\infty}(\pa \TT^2)} \sum\limits_{|\alpha| \le s}\left(\int_{-\pi}^{\pi} \left||D^{\alpha} \widehat{\psi} (x_1, \pi)|^2 - |D^{\alpha} \widehat{\psi}(x_1, -\pi)|^2 \right| dx_1 \right)
\\&\quad\quad\quad+ C\|R^{\perp} \psi\|_{L^{\infty}(\pa \TT^2)} \sum\limits_{|\alpha| \le s} \left(\int_{-\pi}^{\pi} \left||D^{\alpha} \widehat{\psi} (\pi, x_2)|^2 - |D^{\alpha} \widehat{\psi}(-\pi, x_2)|^2\right| dx_2 \right)
\\&\quad\quad\quad\quad+C\|R^{\perp} \psi\|_{L^{\infty}(\pa \TT^2)} \sum\limits_{|\alpha| \le s} \left(\int_{-\pi}^{\pi} |D^{\alpha} \psi (x_1, \pi)| |D^{\alpha} \widehat{\psi}(x_1, \pi) - D^{\alpha} \widehat{\psi}(x_1, - \pi) | \right) 
\\&\quad\quad\quad\quad\quad+C\|R^{\perp} \psi\|_{L^{\infty}(\pa \TT^2)} \sum\limits_{|\alpha| \le s} \left(\int_{-\pi}^{\pi} |D^{\alpha} \psi (\pi, x_2)| |D^{\alpha} \widehat{\psi}(\pi, x_2) - D^{\alpha} \widehat{\psi}(-\pi, x_2) | \right)
\\&\quad\quad\le C\|\psi\|_{H^3(\TT^2)}\left(\|\psi\|_{H^{s+3}(\TT^2)} + \|\widehat{\psi}\|_{H^{s+3}(\TT^2)} \right)\|\mathcal{R}_{b}[s]\|_{L^1(\TT^2)}^{\fr{1}{2}}, 
\end{aligned}
\ee where the last inequality follows from applications of the algebraic inequality $|a^2 - b^2| \le |a-b||a+b|$, the trace theorem, and continuous Sobolev embedding. Moreover, applying Leibniz's rule and H\"older's inequality, we deduce the commutator estimate
\be \nonumber
\|D^{\alpha} (R^{\perp} \psi \cdot \na (\psi - \widehat{\psi})) - R^{\perp} \psi \cdot \na D^{\alpha} (\psi - \widehat{\psi})\|_{L^2(\TT^2)}
\le C\|R^{\perp} \psi\|_{W^{s,\infty}(\TT^2)} \|\psi - \widehat{\psi}\|_{H^s(\TT^2)},
\ee which boils down to 
\be \nonumber
\|D^{\alpha} (R^{\perp} \psi \cdot \na (\psi - \widehat{\psi})) - R^{\perp} \psi \cdot \na D^{\alpha} (\psi - \widehat{\psi})\|_{L^2(\TT^2)}
\le C\|\psi\|_{H^{s+2}(\TT^2)} \|\psi - \widehat{\psi}\|_{H^s(\TT^2)},
\ee after making use of continuous Sobolev embedding and the boundedness of the periodic Riesz transform on Sobolev spaces. Consequently, we infer that 
\be \nonumber
\beg{aligned}
&\left|\sum\limits_{|\alpha| \le s} (D^{\alpha} (R^{\perp} \psi \cdot \na (\psi - \widehat{\psi})), D^{\alpha}(\psi - \widehat{\psi}))_{L^2(\TT^2)} \right|
\le C\|\psi\|_{H^{s+2}(\TT^2)} \|\psi - \widehat{\psi}\|_{H^s(\TT^2)}^2 \\&\quad\quad+ C\|\psi\|_{H^3(\TT^2)}\left(\|\psi\|_{H^{s+3}(\TT^2)} + \|\widehat{\psi}\|_{H^{s+3}(\TT^2)} \right)\|\mathcal{R}_{b}[s]\|_{L^1(\TT^2)}^{\fr{1}{2}} .
\end{aligned}
\ee 
As for the nonlinear term involving the non-periodic operator $\tilde{R}$, we have 
\be \nonumber
\beg{aligned}
&\left|\sum\limits_{|\alpha| \le s}(D^{\alpha} (\tilde{R}^{\perp}(\psi - \widehat{\psi}) \cdot \na \widehat{\psi}), D^{\alpha}(\psi - \widehat{\psi}))_{L^2(\TT^2)} \right|
\\&\quad\quad\le \sum\limits_{|\alpha| \le s} \|D^{\alpha}(\tilde{R}^{\perp} (\psi 
-\widehat{\psi}) \cdot \na \widehat{\psi})\|_{L^2(\TT^2)} \|D^{\alpha} (\psi - \widehat{\psi})\|_{L^2(\TT^2)}
\\&\quad\quad\le C\|\psi - \widehat{\psi}\|_{H^s(\TT^2)} \sum\limits_{|\alpha| \le s} \|\tilde{R}^{\perp} D^{\alpha}\ (\psi - \widehat{\psi}) \|_{L^2(\TT^2)} \|\na \widehat{\psi}\|_{W^{s, \infty}(\TT^2)}.
\end{aligned}
\ee In view of the property \ref{r3} satisfied by $\tilde{R}$, we bound  
\begin{equation*} 
\|\tilde{R}^{\perp} D^{\alpha}\ (\psi - \widehat{\psi}) \|_{L^2(\TT^2)}^2
\le C\|D^{\alpha} \psi -D^{\alpha} \widehat{\psi}\|_{L^2(2\TT^2)}^2
\le C\|\psi - \widehat{\psi}\|_{H^s(\TT^2)}^2 + C\|\mathcal{R}_{per}[s] \|_{L^1(2\TT^2)},
\end{equation*}
and so
\be \nonumber
\beg{aligned}
&\left|\sum\limits_{|\alpha| \le s}(D^{\alpha} (\tilde{R}^{\perp}(\psi - \widehat{\psi})) \cdot \na \widehat{\psi}, D^{\alpha}(\psi - \widehat{\psi}))_{L^2(\TT^2)} \right|
\le C\|\widehat{\psi}\|_{H^{s+3}(\TT^2)} \|\psi - \widehat{\psi}\|_{H^s(\TT^2)}^2
\\&\quad\quad\quad+ C\|\widehat{\psi}\|_{H^{s+3}(\TT^2)} \|\mathcal{R}_{per}[s]\|_{L^1(2\TT^2)}^{\fr{1}{2}}  \|\psi - \widehat{\psi}\|_{H^s(\TT^2)}.
\end{aligned}
\ee These estimates yield the differential inequality 
\beg{align*}
&\fr{d}{dt} \|\psi - \widehat{\psi}\|_{H^s(\TT^2)}^2 
\le C\left(\|\widehat{\psi}\|_{H^{s+3}(\TT^2)}^2 + \|\psi\|_{H^{s+2}(\TT^2)} 
 + 1\right) \|\psi - \widehat{\psi}\|_{H^{s}(\TT^2)}^2 
+ C\|\mathcal{R}_{i}\|_{H^s(\TT^2)}^2
  \\&\quad\quad+ C\|\mathcal{R}_{per}[s]\|_{L^1(2\TT^2)}
+ C\|\psi\|_{H^3(\TT^2)}\left(\|\psi\|_{H^{s+3}(\TT^2)} + \|\widehat{\psi}\|_{H^{s+3}(\TT^2)} \right)\|\mathcal{R}_{b}[s]\|_{L^1(\TT^2)}^{\fr{1}{2}} 
\\&\quad\quad\quad\quad+C\left(1 + \|\psi\|_{W^{s+1, 
 \infty}(\TT^2)} \right)  \|\mathcal{R}_{per}[s]\|_{L^1(2\TT^2)}^{\fr{1}{2}}.
\end{align*}
By Gronwall's inequality, we deduce that
\be \la{total2r}
\beg{aligned}
&\|\psi(t) - \widehat{\psi}(t)\|_{H^s(\TT^2)}^2
\le C_{T,0} \left(\mathcal{E}_{G}[s;\theta]^2 + 
\left(\int_{0}^{t} \|\widehat{\psi}\|_{H^{s+3}(\TT^2)}^2 d\widetilde{t} \right)^{\fr{1}{2}} \mathcal{E}_{G}[s;\theta] \right)
\\&\hspace{5cm}\times e^{\int_{0}^{t} \left(\|\widehat{\psi}\|_{H^{s+3}(\TT^2)} + \|\psi\|_{H^{s+2}(\TT^2)}  +1 \right) d\widetilde{t}}
\end{aligned}
\ee holds for all $t \in [0,T]$. Here $C_{T,0}$ is a constant depending only on $T$ and the size of the initial data $\psi_0$. We bound
{\small\be \la{total1r}
\int_{0}^{t} \|\widehat{\psi}(s)\|_{H^{s+3}(\TT^2)}  ds
\le \sqrt{t} \left( \int_{0}^{t} \|\widehat{\psi}(s)\|_{H^{s+3}(\TT^2)}^2  ds\right)^{\fr{1}{2}}
\le \sqrt{t} \mathcal{E}_{G}^p [s;\theta] 
\le \frac{\sqrt{t} \mathcal{E}_G[s;\theta]}{\sqrt{\lambda}}.
\ee} 
Combining \eqref{total2r} and \eqref{total1r}, we obtain the desired estimate \eqref{total3r}.
\end{proof}

\subsection{Conclusion} We have provided positive answers to both \ref{Q1} and \ref{SubQ2}. These results demonstrate the feasibility of using feedforward neural networks with \texttt{tanh} activation functions to approximate the solution to the SGQ equation. In particular, \ref{Q1} ensures the existence of neural network approximations that nearly satisfy the PDE relation, while \ref{SubQ2} guarantees that these approximations are indeed close to the true solution.

\appendix 

\section{Regularity of Solutions to the SQG Equation} \la{AP}

In this appendix, we address the global smoothness of solutions to the SQG problem \eqref{SQG}--\eqref{initialdata}. We first recall the following theorem regarding the global well-posedness of \eqref{SQG}--\eqref{initialdata} for Sobolev $H^1(\TT^2)$ initial data:

\beg{Thm}[{\cite[Theorem 4.5]{constantin2015long}}] \label{h1regularity} 
Let $T>0$ be arbitrary. Let $\psi_0 \in H^1(\TT^2)$. There exists a unique solution $\psi$ to \eqref{SQG}--\eqref{initialdata} such that
$
\psi \in L^{\infty}(0,T; H^1(\TT^2)) \cap L^2(0,T; H^{\fr{3}{2}}(\TT^2)). $
\end{Thm}

We next present a higher-order space-time regularity of solutions to the 2D problem. 

\beg{Thm} \label{smoothnessSQG}
Let $T>0$ be arbitrary. Let $s > 1$ be an integer. Suppose $\psi_0 \in H^s(\TT^2).$ Then the unique solution $\psi$ to the problem \eqref{SQG}--\eqref{initialdata} satisfies the regularity criterion
\be \label{hsregularity}
\psi \in L^{\infty}(0,T; H^s(\TT^2)) \cap L^2(0,T; H^{s+\fr{1}{2}}(\TT^2)).
\ee Moreover, it holds that
\be \label{hsregularity2}
\psi \in C^{s-2} ([0,T] \times \TT^2).
\ee 
\end{Thm}

\begin{proof} For each $\epsilon > 0$, we consider a parabolic regularization of the SQG equation (diffused by the additional term $-\epsilon \Delta \psi$) with smoothed-out initial data. Each $\epsilon$-regularized system has a unique global smooth solution. Moreover, the family of resulting solutions converges to the unique solution of the initial value problem \eqref{SQG}--\eqref{initialdata}. The following calculations should be performed on the $\epsilon$-regularized systems and inherited to the solution of \eqref{SQG}--\eqref{initialdata} via use of the Banach Alaoglu theorem and the lower semi-continuity of the involved norms. We disregard this technicality for simplicity and directly address the energy evolutions of \eqref{SQG}--\eqref{initialdata}. We also denote the norms $\|\cdot\|_{L^p(\TT^2)}$ by $\|\cdot\|_{L^p}$ as the spatial domain is $\TT^2$ everywhere in the proof below. 

We take the scalar product in $L^2(\TT^2)$ of the SQG equation \eqref{SQG} with $\l^{2s} \psi$ and obtain the energy equation
\be \label{sevolution}
\fr{1}{2} \fr{d}{dt}\|\l^s \psi\|_{L^2}^2 + \|\l^{s + \fr{1}{2}}\psi\|_{L^2}^2 = -\int_{\TT^2} u \cdot \na \psi \l^{2s} \psi dx.
\ee
Integrating by parts the nonlinear term, and using the divergence-free condition obeyed by $u$ gives the relation
\be 
\int_{\TT^2} u \cdot \na \psi \l^{2s} \psi dx
= \int_{\TT^2} (\l^s (u \cdot \na \psi) - u \cdot \na \l^s \psi) \l^s \psi dx. \nonumber
\ee 
Due to the periodic commutator estimate (see \cite[Appendix~A]{constantin2015long})
\be 
\|\l^s (u \cdot \na \psi) - u \cdot \l^s \na \psi \|_{L^2} 
\le C\|\na u \|_{L^4} \|\l^{s-1} \na \psi\|_{L^4} + C\|\l^s u \|_{L^4} \|\na \psi\|_{L^4}, \nonumber
\ee the fact that $\l^s$ and $\na$ commutes (as these operators  are Fourier multipliers), the continuous embedding of $H^{\fr{1}{2}}(\TT^2)$ in $L^4(\TT^2)$, and the boundedness of the Riesz transform on the periodic Sobolev spaces, we have 
\be 
\beg{aligned}
\left|\int_{\TT^2} u \cdot \na \psi \l^{2s} \psi dx\right|
\le C\|\l^{\fr{3}{2}} \psi\|_{L^2} \|\l^{s+\fr{1}{2}} \psi\|_{L^2} \|\l^s \psi\|_{L^2}. \nonumber
\end{aligned}
\ee 
A straightforward application of Young's inequality for products and exploitation of the dissipative structure of \eqref{sevolution}  give rise to the differential inequality
\be 
\fr{d}{dt} \|\l^s \psi\|_{L^2}^2
+ \|\l^{s+\fr{1}{2}}\psi\|_{L^2}^2
\le C\|\l^{\fr{3}{2}} \psi\|_{L^2}^2 \|\l^s \psi\|_{L^2}^2. \nonumber
\ee 
By Theorem \ref{h1regularity} and Gronwall inequality, we deduce \eqref{hsregularity}. 

Due to the fractional product estimate
\be 
\|\l^{s} (u\psi)\|_{L^2}
\le C\|u\|_{L^{\infty}}\|\l^s \psi\|_{L^2} + C\|\l^s u\|_{L^2} \|\psi\|_{L^{\infty}}, \nonumber
\ee the continuous embedding of $H^{\fr{3}{2}}(\TT^2)$ in $L^{\infty}(\TT^2)$, the boundedness of the Riesz transform on $H^{\fr{3}{2}}(\TT^2)$ and $H^s(\TT^2)$, and the fact that $R^{\perp} \psi$ is divergence-free, we can estimate the nonlinear term $u \cdot \na \psi$ in the Sobolev norm of $H^{s -1}(\TT^2)$ and obtain 
\be \label{est:reg-1}
\|\l^{s-1} (u \cdot \na \psi)\|_{L^2}
= \|\l^{s-1} \na \cdot (u \psi)\|_{L^2}
\le C\|\l^s (u\psi)\|_{L^2}
\le C\|\l^s \psi\|_{L^2} \|\l^{\fr{3}{2}} \psi\|_{L^2}.
\ee Thus, it holds that
\be 
\|\l^{s-1}(\pa_t \psi)\|_{L^2}
\le \|\l^s \psi\|_{L^2} + C\|\l^s \psi\|_{L^2} \|\l^{\fr{3}{2}} \psi\|_{L^2}, \nonumber
\ee from which we infer that
$\pa_t \psi \in L^2(0,T; H^{s-1}(\TT^2))$. Now we apply the Aubin-Magenes lemma and deduce that $\psi \in C(0,T; H^s(\TT^2))$. 
As $s\geq 2$, it follows from \eqref{est:reg-1} that $u\cdot\nabla \psi \in C(0,T; H^{s-1}(\TT^2))$.
 Since the diffusion term $\l \psi$ belongs to $C(0,T; H^{s-1}(\TT^2))$, we infer that
$\pa_t \psi \in C(0,T; H^{s-1}(\TT^2))$, and so $ \psi \in C^1(0,T; H^{s-1}(\TT^2))$.
We bootstrap and obtain the regularity $\psi \in \bigcap_{k=0}^{s-2} C^{k}(0,T; H^{s-k}(\TT^2))$. As $H^{s-k}(\TT^2)$ is continuously embedded in $C^{s-k-2}(\TT^2)$, it follows that $\psi \in \bigcap_{k=0}^{s-2} C^{k}(0,T; C^{s-k-2}(\TT^2))$, yielding the desired smoothness property \eqref{hsregularity2}. 
\end{proof}

\section*{Acknowledgments}

R.H. was partially supported by a grant from the Simons Foundation (MP-TSM-00002783). Q.L. was partially supported by the AMS-Simons Travel Grant.

\vspace{0.5cm}

{\bf{Data Availability Statement.}} The research does not have any associated data.

\bibliographystyle{siamplain}
\bibliography{reference}

\begin{thebibliography}{10}

\bibitem{beck2019machine}
{\sc C.~Beck, W.~E, and A.~Jentzen}, {\em Machine learning approximation algorithms for high-dimensional fully nonlinear partial differential equations and second-order backward stochastic differential equations}, Journal of Nonlinear Science, 29 (2019), pp.~1563--1619.

\bibitem{biswas2022error}
{\sc A.~Biswas, J.~Tian, and S.~Ulusoy}, {\em Error estimates for deep learning methods in fluid dynamics}, Numerische Mathematik, 151 (2022), pp.~753--777.

\bibitem{bonito2021numerical}
{\sc A.~Bonito and M.~Nazarov}, {\em Numerical simulations of surface quasi-geostrophic flows on periodic domains}, SIAM Journal on Scientific Computing, 43 (2021), pp.~B405--B430.

\bibitem{caffarelli2010drift}
{\sc L.~A. Caffarelli and A.~Vasseur}, {\em Drift diffusion equations with fractional diffusion and the quasi-geostrophic equation}, Annals of Mathematics,  (2010), pp.~1903--1930.

\bibitem{castro2009incompressible}
{\sc {\'A}.~Castro, D.~C{\'o}rdoba, F.~Gancedo, and R.~Orive}, {\em Incompressible flow in porous media with fractional diffusion}, Nonlinearity, 22 (2009), p.~1791.

\bibitem{constantin2016critical}
{\sc P.~Constantin and M.~Ignatova}, {\em Critical {SQG} in bounded domains}, Annals of PDE, 2 (2016), p.~8.

\bibitem{constantin2020estimates}
{\sc P.~Constantin and M.~Ignatova}, {\em Estimates near the boundary for critical {SQG}}, Annals of PDE, 6 (2020), pp.~1--30.

\bibitem{constantin2012new}
{\sc P.~Constantin, M.-C. Lai, R.~Sharma, Y.-H. Tseng, and J.~Wu}, {\em New numerical results for the surface quasi-geostrophic equation}, Journal of Scientific Computing, 50 (2012), pp.~1--28.

\bibitem{constantin1994formation}
{\sc P.~Constantin, A.~J. Majda, and E.~Tabak}, {\em Formation of strong fronts in the 2-d quasigeostrophic thermal active scalar}, Nonlinearity, 7 (1994), p.~1495.

\bibitem{constantin1994singular}
{\sc P.~Constantin, A.~J. Majda, and E.~G. Tabak}, {\em Singular front formation in a model for quasigeostrophic flow}, Physics of Fluids, 6 (1994), pp.~9--11.

\bibitem{constantin2018local}
{\sc P.~Constantin and H.~Q. Nguyen}, {\em Local and global strong solutions for {SQG} in bounded domains}, Physica D: Nonlinear Phenomena, 376 (2018), pp.~195--203.

\bibitem{constantin2015long}
{\sc P.~Constantin, A.~Tarfulea, and V.~Vicol}, {\em Long time dynamics of forced critical {SQG}}, Communications in Mathematical Physics, 335 (2015), pp.~93--141.

\bibitem{constantin2012nonlinear}
{\sc P.~Constantin and V.~Vicol}, {\em Nonlinear maximum principles for dissipative linear nonlocal operators and applications}, Geometric And Functional Analysis, 22 (2012), pp.~1289--1321.

\bibitem{constantin1999behavior}
{\sc P.~Constantin and J.~Wu}, {\em Behavior of solutions of 2d quasi-geostrophic equations}, SIAM journal on mathematical analysis, 30 (1999), pp.~937--948.

\bibitem{cordoba2004maximum}
{\sc A.~C{\'o}rdoba and D.~C{\'o}rdoba}, {\em A maximum principle applied to quasi-geostrophic equations}, Communications in mathematical physics, 249 (2004), pp.~511--528.

\bibitem{cuomo2022scientific}
{\sc S.~Cuomo, V.~S. Di~Cola, F.~Giampaolo, G.~Rozza, M.~Raissi, and F.~Piccialli}, {\em Scientific machine learning through physics-informed neural networks: Where we are and what's next}, arXiv preprint arXiv:2201.05624,  (2022).

\bibitem{de2023error}
{\sc T.~De~Ryck, A.~D. Jagtap, and S.~Mishra}, {\em {Error estimates for physics-informed neural networks approximating the Navier-Stokes equations}}, IMA Journal of Numerical Analysis,  (2023).

\bibitem{de2022error}
{\sc T.~De~Ryck and S.~Mishra}, {\em Error analysis for physics-informed neural networks (pinns) approximating kolmogorov pdes}, Advances in Computational Mathematics, 48 (2022), p.~79.

\bibitem{de2022generic}
{\sc T.~De~Ryck and S.~Mishra}, {\em Generic bounds on the approximation error for physics-informed (and) operator learning}, Advances in Neural Information Processing Systems, 35 (2022), pp.~10945--10958.

\bibitem{dissanayake1994neural}
{\sc M.~Dissanayake and N.~Phan-Thien}, {\em Neural-network-based approximations for solving partial differential equations}, communications in Numerical Methods in Engineering, 10 (1994), pp.~195--201.

\bibitem{yu2018deep}
{\sc W.~E and B.~Yu}, {\em The deep {R}itz method: a deep learning-based numerical algorithm for solving variational problems}, Communications in Mathematics and Statistics, 6 (2018), pp.~1--12.

\bibitem{HaJeE:18}
{\sc J.~Han, A.~Jentzen, and W.~E}, {\em Solving high-dimensional partial differential equations using deep learning}, Proceedings of the National Academy of Sciences, 115 (2018), pp.~8505--8510.

\bibitem{held1995surface}
{\sc I.~M. Held, R.~T. Pierrehumbert, S.~T. Garner, and K.~L. Swanson}, {\em Surface quasi-geostrophic dynamics}, Journal of Fluid Mechanics, 282 (1995), pp.~1--20.

\bibitem{higham2019deep}
{\sc C.~F. Higham and D.~J. Higham}, {\em Deep learning: An introduction for applied mathematicians}, Siam review, 61 (2019), pp.~860--891.

\bibitem{hu2023higher}
{\sc R.~Hu, Q.~Lin, A.~Raydan, and S.~Tang}, {\em Higher-order error estimates for physics-informed neural networks approximating the primitive equations}, Partial Differential Equations and Applications, 4 (2023), p.~34.

\bibitem{ignatova2019construction}
{\sc M.~Ignatova}, {\em Construction of solutions of the critical {SQG} equation in bounded domains}, Advances in Mathematics, 351 (2019), pp.~1000--1023.

\bibitem{jagtap2021extended}
{\sc A.~D. Jagtap and G.~E. Karniadakis}, {\em Extended physics-informed neural networks (xpinns): A generalized space-time domain decomposition based deep learning framework for nonlinear partial differential equations.}, in AAAI spring symposium: MLPS, vol.~10, 2021.

\bibitem{jagtap2020conservative}
{\sc A.~D. Jagtap, E.~Kharazmi, and G.~E. Karniadakis}, {\em Conservative physics-informed neural networks on discrete domains for conservation laws: Applications to forward and inverse problems}, Computer Methods in Applied Mechanics and Engineering, 365 (2020), p.~113028.

\bibitem{karniadakis2021physics}
{\sc G.~E. Karniadakis, I.~G. Kevrekidis, L.~Lu, P.~Perdikaris, S.~Wang, and L.~Yang}, {\em Physics-informed machine learning}, Nature Reviews Physics, 3 (2021), pp.~422--440.

\bibitem{kiselev2007global}
{\sc A.~Kiselev, F.~Nazarov, and A.~Volberg}, {\em Global well-posedness for the critical 2 d dissipative quasi-geostrophic equation}, Inventiones mathematicae, 167 (2007), pp.~445--453.

\bibitem{lagaris1998artificial}
{\sc I.~E. Lagaris, A.~Likas, and D.~I. Fotiadis}, {\em Artificial neural networks for solving ordinary and partial differential equations}, IEEE transactions on neural networks, 9 (1998), pp.~987--1000.

\bibitem{lagaris2000neural}
{\sc I.~E. Lagaris, A.~C. Likas, and D.~G. Papageorgiou}, {\em Neural-network methods for boundary value problems with irregular boundaries}, IEEE Transactions on Neural Networks, 11 (2000), pp.~1041--1049.

\bibitem{lu2021deepxde}
{\sc L.~Lu, X.~Meng, Z.~Mao, and G.~E. Karniadakis}, {\em {DeepXDE}: A deep learning library for solving differential equations}, SIAM Review, 63 (2021), pp.~208--228, \url{https://doi.org/10.1137/19M1274067}.

\bibitem{mao2020physics}
{\sc Z.~Mao, A.~D. Jagtap, and G.~E. Karniadakis}, {\em Physics-informed neural networks for high-speed flows}, Computer Methods in Applied Mechanics and Engineering, 360 (2020), p.~112789.

\bibitem{mishra2022estimates1}
{\sc S.~Mishra and R.~Molinaro}, {\em Estimates on the generalization error of physics-informed neural networks for approximating a class of inverse problems for pdes}, IMA Journal of Numerical Analysis, 42 (2022), pp.~981--1022.

\bibitem{mishra2022estimates}
{\sc S.~Mishra and R.~Molinaro}, {\em Estimates on the generalization error of physics-informed neural networks for approximating pdes}, IMA Journal of Numerical Analysis, 43 (2022), pp.~1--43.

\bibitem{ohkitani2012asymptotics}
{\sc K.~Ohkitani}, {\em Asymptotics and numerics of a family of two-dimensional generalized surface quasi-geostrophic equations}, Physics of Fluids, 24 (2012).

\bibitem{raissi2018hidden}
{\sc M.~Raissi and G.~E. Karniadakis}, {\em Hidden physics models: Machine learning of nonlinear partial differential equations}, Journal of Computational Physics, 357 (2018), pp.~125--141.

\bibitem{raissi2019physics}
{\sc M.~Raissi, P.~Perdikaris, and G.~E. Karniadakis}, {\em Physics-informed neural networks: A deep learning framework for solving forward and inverse problems involving nonlinear partial differential equations}, Journal of Computational physics, 378 (2019), pp.~686--707.

\bibitem{sirignano2018dgm}
{\sc J.~Sirignano and K.~Spiliopoulos}, {\em Dgm: A deep learning algorithm for solving partial differential equations}, Journal of computational physics, 375 (2018), pp.~1339--1364.

\bibitem{song2017fractional}
{\sc F.~Song and G.~E. Karniadakis}, {\em Fractional spectral vanishing viscosity method: Application to the quasi-geostrophic equation}, Chaos, Solitons \& Fractals, 102 (2017), pp.~327--332.

\bibitem{stein1970singular}
{\sc E.~M. Stein}, {\em Singular integrals and differentiability properties of functions}, Princeton university press, 1970.

\bibitem{stokols2020holder}
{\sc L.~F. Stokols and A.~F. Vasseur}, {\em H{\"o}lder regularity up to the boundary for critical {SQG} on bounded domains}, Archive for Rational Mechanics and Analysis, 236 (2020), pp.~1543--1591.

\bibitem{wu2014well}
{\sc J.~Wu and X.~Xu}, {\em Well-posedness and inviscid limits of the boussinesq equations with fractional laplacian dissipation}, Nonlinearity, 27 (2014), p.~2215.

\bibitem{wu20182d}
{\sc J.~Wu, X.~Xu, and Z.~Ye}, {\em The 2d boussinesq equations with fractional horizontal dissipation and thermal diffusion}, Journal de Math{\'e}matiques Pures et Appliqu{\'e}es, 115 (2018), pp.~187--217.

\bibitem{yang2014global}
{\sc W.~Yang, Q.~Jiu, and J.~Wu}, {\em Global well-posedness for a class of 2d boussinesq systems with fractional dissipation}, Journal of Differential Equations, 257 (2014), pp.~4188--4213.

\bibitem{yang20183d}
{\sc W.~Yang, Q.~Jiu, and J.~Wu}, {\em The 3d incompressible boussinesq equations with fractional partial dissipation}, Communications in Mathematical Sciences, 16 (2018), pp.~617--633.

\end{thebibliography}

\end{document}